\documentclass{article}
\usepackage{microtype}
\usepackage{graphicx}
\usepackage{subfigure}
\usepackage{booktabs}
\usepackage{hyperref}
\usepackage{xcolor}
\usepackage{fullpage}

\usepackage{amssymb,amsmath,amsthm}
\usepackage{empheq}
\usepackage{algorithm,algorithmic}
\usepackage{times}

\newcommand{\ignore}[1]{}

\def\reals{{\mathbb R}}

\newcommand{\ftil}{\tilde{f}}

\newcommand{\xbar}{\bar{\mathbf{x}}}

\def\x{\mathbf{x}}
\def\y{\mathbf{y}}
\def\z{\mathbf{z}}

\def\epsilon{\varepsilon}

\newtheorem{Thm}{Theorem}

\newtheorem{theorem}{Theorem}
\newtheorem{lemma}[Thm]{Lemma}

\DeclareMathOperator*{\argmin}{\arg\!\min}

\def\citet{\cite}

\begin{document}
\twocolumn
\date{}
\title{Revisiting the Polyak Step Size}
\author{ Elad Hazan \thanks{Google AI Princeton} \thanks{Princeton University} \and Sham M. Kakade \footnotemark[1] \thanks{University of Washington  } 
}\maketitle

\begin{abstract}
This note revisits the Polyak step size schedule for convex
optimization problems,  proving that a simple variant of it
simultaneously attains near optimal convergence rates for the gradient
descent algorithm, for all  ranges
of strong convexity, smoothness, and Lipschitz parameters, without
a priori knowledge of these parameters. 
\end{abstract}

\section{Introduction}

Scaleable optimization for machine learning is based entirely on first
order gradient methods. Besides the age-old method of stochastic
approximation \cite{robbins1985stochastic}, three accelerated methods
have proved their practical and theoretical significance: Nesterov
acceleration \cite{Nesterov83}, variance reduction \cite{schmidt2017minimizing} and
adaptive learning-rate/regularization \cite{duchi2011adaptive}.

Adaptive choices of step sizes allow optimization algorithms to
accelerate quickly according to the local curvature and smoothness of
the optimization landscape.  However, in theory, there are few
parameter free algorithms, and, in practice, there are many search
heuristics utilized.

Let us examine this question of parameter free, adaptive learning
rates for one of the most standard algorithms, namely the 
gradient descent method: 
\begin{equation}\label{eq:gd}
  \x_{t+1}  =  \x_t - \eta_t \nabla f(\x_t)  \, .
\end{equation}
Although
this class of algorithms is not optimal in all settings (i.e. the
aforementioned accelerations can be applied), it is fundamental, and
we may ask what are optimal known rates along with the optimal step
size choices are for this particular algorithm.  Here,
Table~\ref{table:offline} shows the best known rates for 
gradient descent in the standard
regimes: general convex (non-smooth with bounded sub-gradients);
$\beta$-smooth; $\alpha$-strongly-convex; and
$\beta$-smooth\&$\alpha$-strongly convex (see
\cite{boyd2004convex,bubeck2015convex} for more details). 

From a practical perspective these step size settings are
unfortunately disparate in  various regimes: ranging from
rapidly decaying at $\eta_t=O( \frac{1}{\alpha t})$ to moderately decaying at
$\eta_t=O( \frac{1}{\sqrt{t}})$ to a constant
$\eta_t=\frac{1}{\beta}$ (see \cite{boyd2004convex,bubeck2015convex} for more details). 

{\bf This work:} 
We show that a single (and simple) choice of a step size schedule
gives, simultaneously, the optimal convergence (among the class of
 gradient descent algorithms) in all these regimes, without
knowing these parameters in advance. Perhaps surprisingly, this choice
is that prescribed by \citet{polyak}, who argued that this choice was
optimal for the non-smooth, convex case (marked as ``convex" in
Table~\ref{table:offline}, see also \cite{BoydNotes}).


\begin{table}[t]
\begin{center} 
\begin{tabular}{c|c|c|c|c}
&  convex & $\beta$-smooth & $\alpha$-strongly & $(\alpha,\beta)$-well\\ 
&  &     		& convex			 & conditioned\\ 
 \hline 
  error &$\frac{1}{\sqrt{T}}$ & $\frac{\beta}{T}$
& $\frac{1}{\alpha T}$ & $ e^{-\frac{\beta}{\alpha}T } $\\
   \hline 
step size  &$\frac{1}{\sqrt{T}}$ & $\frac{1}{\beta}$ & $\frac{1}{\alpha
T}$ & $ \frac{1}{\beta}$ \\ 
  \hline
\end{tabular}
\end{center}
\caption{Standard convergence rates of gradient descent in convex
  optimization problems.  Error denotes $f(\x_t) - f(\x^{\star})$  of
  a first order methods as a function of the number of iterations. Step Size is
  the standard learning rate schedule used to obtain this rate. Dependence on
  other parameters, namely the Lipchitz constant and initial distance to the objective, is
  omitted. } \label{table:GD} 
\label{table:offline}
\end{table}%

\section{Convexity Preliminaries}

We  consider the minimization of a continuous convex
function over Euclidean space $f: \reals^d \mapsto \reals$ by an
iterative gradient-based method.  We say that $f$ is $\alpha$-strongly
convex if and only if $\forall \x,\y$: 
$$ f(\y) \geq f(\x) + \nabla f(\x)(\y-\x) + \frac{\alpha}{2} \|\x-\y\|^2 $$
We say that $f$ is $\beta$ smooth if and only if $\forall \x,\y$:
$$ f(\y) \leq f(\x) + \nabla f(\x)(\y-\x) + \frac{\beta}{2} \|\x-\y\|^2. $$
The following notation is used throughout: 
\begin{itemize}
\item
$\x^{\star} = \argmin_{\x \in \reals^d} \left\{ f(\x) \right\}$ - optimum
\item
$h(\x_t) = h_t = f(\x_t) - f(\x^{\star})$ - sub-optimality gap of the iterate
\item
$d_t = \|\x_t - \x^{\star}\| $ - Euclidean distance of the iterate.
\item
$\nabla_t = \nabla f(\x_t)$ - gradient of the iterate.
\item
$ \|\nabla_t\|^2$ denotes squared Euclidean norm.
\end{itemize}

The following are basic properties for $\alpha$-strongly-convex functions and/or $\beta$-smooth functions (proved for completeness in Lemma \ref{lem:elementary_properties}):
\begin{equation}\label{eqn:shalom2}
\frac{\alpha}{2} d_t^2 \leq h_t \leq \frac{\beta}{2} d_t^2 \  \ , \ \  \frac{1}{2\beta} \|\nabla_t\|^2 \leq h_t \leq \frac{1}{2\alpha} \|\nabla_t\|^2 
\end{equation}
and thus, 
$$ \frac{1}{4 \beta^2}  \|\nabla_t\|^2 \leq d_t^2 \leq  \frac{1}{4 \alpha^2}
\|\nabla_t\|^2  \, .$$

The following standard lemma is at the heart of much of the analysis of first
order convex optimization.
\begin{lemma} \label{lemma:shalom1}
The sequence of iterates produced by projected gradient descent
(equation~\ref{eq:gd}) satisfies: 
\begin{equation} \label{eqn:shalom}
d_{t+1}^2 \leq d_t^2 - 2 \eta_t h_t + \eta_t^2 \|\nabla_t\|^2 
\end{equation}
\end{lemma}
\begin{proof}
By algorithm definition we have,
\begin{eqnarray*}
d_{t+1}^2 &=& \|\x_{t+1} - \x^{\star}\|^2\\
& = &   \|  \x_t   - \eta_t \nabla_t - \x^{\star} \|^2  \\
& = &  d_t^2 - 2 \eta_t \nabla_t^\top(\x_t - \x^{\star}) + \eta_t^2 \|\nabla_t\|^2 \\
& \leq&   d_t^2 - 2 \eta_t h_t + \eta_t^2 \|\nabla_t\|^2   \label{eq:convexity} 
\end{eqnarray*}
where we have used properties of 
convexity in the last step.
\end{proof}

\section{Main Results}

\citet{polyak} argued that, in a sense, the optimal step size choice of
$\eta_t$ should decrease the upper bound on $d_{t+1}^2$ as fast as
possible.  This choice is:
\[
\eta_t = \frac{h_t}{\|\nabla_t\|^2}
\]
which leads to a decrease of $d_t^2$ by:
\[
d_{t+1}^2 \leq d_t^2 -  \frac{ h_t^2}{\|\nabla_t\|^2}
\]
Note that this choice utilizes knowledge of $f(\x^{\star})$, since $h_t =
f(\x_t) - f(\x^{\star})$.  

\citet{polyak} showed that this choice was optimal for non-smooth
convex optimization (i.e. for bounded gradients). Our first result
shows that this step size schedule (which knows $f(\x^{\star})$) achieves the
min of the best known bounds in all the standard parameter regimes
(among the class of projected gradient descent algorithms).  Assume $\|\nabla_t\| \leq G$, and define:
\begin{eqnarray*}
  B_T 
      &=&  \min\left\{
  \frac{G d_0}{\sqrt{ T}},
  \frac {2\beta d_0^2}{ T },
  \frac{G^2}{  \alpha  T } ,
   \beta d_0^2\left(1-\frac{\alpha}{2\beta}\right)^T
 \right\}.
\end{eqnarray*}

\begin{theorem} \label{thm:simple}
(GD with the Polyak Step Size) 
Algorithm \ref{alg:basic} attains the following regret bound after $T$ steps: 
\begin{eqnarray*}
f(\xbar) - f(\x^{\star}) \leq  B_T
\end{eqnarray*}
\end{theorem}

\begin{algorithm}[t!]
\caption{GD with the Polyak stepsize}
\label{alg:basic}
\begin{algorithmic}[1]
\STATE Input: time horizon $T$, $x_0$
\FOR{$t = 0, \ldots, T-1$}
\STATE Set $\eta_t =  \frac{h_t}{\|\nabla_t\|^2} $
\STATE  $ \x_{t+1}  =   \x_t - \eta_t \nabla_t $
\ENDFOR
\STATE Return $\xbar=\x_{t^\star}$ where $t^\star = \argmin_{t < T} \{ f(\x_t) \}$.
\end{algorithmic}
\end{algorithm}

Without knowledge of the optimal function value $f(\x^{\star})$, our second main result
shows that all we need is a lower bound $\ftil_0 \leq f(\x^{\star})$, and we
can do nearly as well as the exact Polyak step size method (up to a
$\log$ factor in $f(\x^{\star})-\ftil_0$).
Note that it is often the case that $\ftil_0 =0$ is a valid lower
bound (e.g. in empirical risk  minimization settings). 

\begin{theorem}\label{thm:simple2}
(The Adaptive Polyak Step Size)   Assume a lower
bound $\ftil_0 \leq  f(\x^{\star})$; that 
$K = 1+\lceil 2 \log \frac{f(\x^{\star})-\ftil_0}{B_T}\rceil$.
Algorithm~\ref{alg:full} returns an $\xbar$ such that:
\[
f(\xbar) - f(\x^{\star}) \leq 2 B_T
\]
Furthermore, the number of gradient descent updates made by the algorithm is at most 
$T \cdot (1+\lceil 2 \log \frac{f(\x^{\star})-\ftil_0}{B_T}\rceil)$.
\end{theorem}

In other words, this algorithms makes at most
$O(T \cdot \log \frac{f(\x^{\star})-\ftil_0}{B_T})$ gradient updates to get $B_T$
error,
while the exact Polyak stepsize uses $T$ updates (to obtain $B_T$
error).  The subtlety in the construction is that even with a initial
lower bound on $\ftil_0$, the values $f(\x_t)$ are only upper
bounds. However, Algorithm~\ref{alg:full} and its proof shows how
either the lower bound can be refined or, if not, the algorithm will
succeed. Note that Algorithm~\ref{alg:full} always call the
subroutine Algorithm~\ref{alg:basic2} starting at the same $x_0$.
 
\subsection{Analysis: the exact case}

\begin{algorithm}[t!]
\caption{GD with a lower bound}
\label{alg:basic2}
\begin{algorithmic}[1]
\STATE Input: time horizon $T$, $\x_0$, lower bound $\ftil \leq  f(\x^{\star})$. 
\FOR{$t = 0, \ldots, T-1$}
\STATE Set $\eta_t =  \frac{f(\x_t) - \ftil}{2 \|\nabla_t\|^2} $
\STATE  $ \x_{t+1}  =  \x_t - \eta_t \nabla_t $
\ENDFOR
\STATE Return $\xbar=\x_{t^\star}$ where $t^\star = \argmin_{t < T} \{ f(\x_t) \}$.
\end{algorithmic}
\end{algorithm}

\begin{algorithm}[t!]
\caption{Adaptive Polyak}
\label{alg:full}
\begin{algorithmic}[1]
\STATE Input: time horizon $T$, number of epochs $K$, $\x_0$, value $\ftil_0 \leq  f(\x^{\star})$. 
\FOR{epoch $k = 0, \ldots, K-1$}
\STATE Let $\bar{\x}_k$ be the output of Algorithm \ref{alg:basic2} using input $\x_0,T,\tilde{f}_k$.
\STATE Update $\tilde{f}_{k+1} \leftarrow
\frac{f(\xbar_k)+\tilde{f_k}}{2}$
\ENDFOR
\STATE Return $\xbar_{k^\star}$ where $k^\star = \argmin_{k < K} \{ f(\xbar_k) \}$.
\end{algorithmic}
\end{algorithm}

Theorem~\ref{thm:simple} directly follows from the following lemma. It
is helpful for us to state this lemma in a more general form,
where, for $0\leq \gamma \leq 1 $,  we define $R_{T,\gamma}$ as follows:
\[
  R_{T,\gamma} = \min\left\{
  \frac{G d_0}{\sqrt{\gamma T}},
  \frac {2 \beta d_0^2}{\gamma T },
  \frac{ G^2}{{\gamma}  \alpha  T } ,
   \beta d_0^2\left(1-\gamma\frac{\alpha}{\beta}\right)^T
 \right\}
.
\]

\begin{lemma} \label{lemma:shalom2}
For $0\leq \gamma \leq 1 $, suppose that a sequence $\x_0,
\ldots \x_t$ satisfies:
\begin{equation} \label{eqn:shalom3}
d_{t+1}^2 \leq d_t^2 -  \gamma \frac{ h_t^2}{\|\nabla_t\|^2}
\end{equation}
then for $\xbar=\x_{t^\star}$, where $t^\star = \argmin_{t < T} \{ f(\x_t) \}$,
\[
h(\xbar)  \leq   R_{T,\gamma}\, .
\]
\end{lemma}

\begin{proof}
The proof analyzes different cases:
\begin{enumerate}
\item
For convex functions with gradient bound $G$, 
\begin{eqnarray*}
d_{t+1}^2 -  d_t^2 & \leq - \frac{\gamma h_t^2}{\|\nabla_t\|^2} \leq -
                     \frac{\gamma h_t^2}{G^2}  
\end{eqnarray*}
Summing up over $T$ iterations, and using Cauchy-Schwartz, we have
\begin{eqnarray*}
\frac{1}{T} \sum_t h_t 
& \leq&  \frac{1}{\sqrt{T}} \sqrt{\sum_t h_t^2} \\
& \leq& \frac{ G}{\sqrt{\gamma T}} \sqrt{\sum_t (d_{t}^2 - d_{t+1}^2)} \leq
\frac{ G d_0 }{\sqrt{\gamma T}} \, .
\end{eqnarray*}

\item
For smooth functions, equation~\eqref{eqn:shalom2} implies:
\[
d_{t+1}^2 - d_t^2 \leq - \frac{\gamma h_t^2}{\|\nabla_t\|^2} \leq -
\frac{\gamma h_t}{2 \beta} \, .
\]
This implies
\[
\frac{1}{T} \sum_t h_t  \leq \frac{2 \beta d_0^2}{\gamma T}\, .
\]

\item
For strongly convex functions, equation~\eqref{eqn:shalom2} implies:
\[ d_{t+1}^2 - d_t^2
  \leq - \gamma \frac{h_t^2}{\|\nabla_t\|^2}
  \leq - \gamma \frac{h_t^2}{G^2}
  \leq  - \gamma  \frac{\alpha^2 d_t^4 }{4 G^2} \, .
\]
In other words,
$d_{t+1}^2  \leq  d_t^2 ( 1- \gamma \frac{\alpha^2 d_t^2}{4 G^2} ) \, .$ 
Defining $a_t :={\gamma}\frac{4 \alpha^2 d_t^2}{G^2}$, we have:
\[
a_{t+1}  \leq  a_t (1-a_t) \, .
\]
This implies that $a_t \leq \frac{1}{t+1}$, which can be seen by
induction\footnote{That $a_0\leq 1$ follows from equation
  \eqref{eqn:shalom2}. For $t=1$, $a_1\leq \frac{1}{2}$ since
  $a_1  \leq  a_0 (1-a_0)$ and $0\leq a_0\leq 1$.
  For the induction step, $
  a_t  \leq  a_{t-1} (1-a_{t-1}) \leq
  \frac{1}{t}(1-\frac{1}{t})
  =\frac{t-1}{t^2}=\frac{1}{t+1}(\frac{t^2-1}{t^2})
  \leq \frac{1}{t+1}$.}. The proof is completed as follows\footnote{This assumes $T$ is even. $T$ odd
    leads to the same constants.} :  
\begin{eqnarray*}
\frac{1}{ T/2 } \sum_{t= T/2 }^T h_t^2 &
\leq& \frac{2G^2}{\gamma  T  }\sum_{t= T/2 }^T ( d_t^2 -
                                   d_{t+1}^2)  \\
  &=&\frac{2 G^2}{\gamma  T } ( d _{ T/2 }^2 - d_T^2)  \\
  &=&\frac{G^4}{2 \gamma^2 \alpha^2   T} ( a
    _{ T/2 } - a_T)  \\ 
   & \leq &\frac{ G^4}{\gamma^2 \alpha^2 T ^2} 
  \, .
\end{eqnarray*}
Thus, there exists a $t$ for which $h_t^2 \leq \frac{ G^4}{\gamma^2 \alpha^2
   T^2} $. Taking the square root completes the claim.

\item
For both strongly convex and smooth: 
\[ d_{t+1}^2 - d_t^2 \leq - \gamma \frac{h_t^2}{\|\nabla_t\|^2} \leq
 - \frac{\gamma h_t}{2 \beta} \leq
  - \gamma \frac{\alpha}{\beta} d_t^2
  \]
  Thus,
  \[
    h_{T} \leq \beta d_{T}^2 \leq \beta d_0^2
    \left(1-\gamma\frac{\alpha}{\beta}\right)^T\, .
    \]
  \end{enumerate}
This completes the proof of all cases.
\end{proof}

\subsection{Analysis: the adaptive case}


The proof of Theorem~\ref{thm:simple2} rests on the following lemma
which shows that, given  a lower bound on the objective, the
subroutine in Algorithm \ref{alg:full} either returns a near-optimal
point with desired precision or a tighter lower bound.

\begin{lemma} \label{lemma:shalom4}
Assume $\|\nabla_t\| \leq G$. 
With input $T$, $\x_0$, and $\ftil$ where $\ftil \leq  f(\x^{\star})$, Algorithm \ref{alg:basic2} returns a point $\xbar$ such that one of the following holds:
\begin{enumerate}
    \item $ h(\xbar) \leq R_{T,\frac{1}{2}} $
    \item For $\tilde{f}_+ := \frac{f(\xbar)+\tilde{f}}{2}$,
      \[ 0 \leq f(\x^{\star}) - \tilde{f}_+  \leq \frac{f(\x^{\star}) - \ftil}{2}  \]
\end{enumerate}
\end{lemma}

\begin{proof}
Due to that $\tilde f$ is a lower bound,  we have that 
  \[
\eta_t =\frac{f(\x_t) - \tilde f}{2 ||\nabla_t||^2} \geq \frac{h_t}{2||\nabla_t||^2} \, .
\]
We will consider two cases. First, suppose that
\begin{equation}\label{eq:condition}
 \eta_t \leq \frac{h_t}{\|\nabla_t\|^2}
\end{equation}
held for $T$ steps. For this case, by Lemma~\ref{lemma:shalom1},
\begin{eqnarray*}
  d_{t+1}^2 &\leq &d_t^2 - 2 \eta_t h_t + \eta_t^2 \|\nabla_t\|^2\\
            &\leq &d_t^2 - 2 \eta_t h_t + \eta_t h_t\\
            &= &d_t^2 -  \eta_t h_t \\
              & \leq &d_t^2 -  \frac{h_t^2}{2||\nabla_t||^2} 
\end{eqnarray*}
using the assumed upper bound on $\eta_t$ in the second step and the
lower bound in the last step.
By Lemma~\ref{lemma:shalom2}, we can take $\gamma=1/2$ and we have
that  $\min_{t< T} h_t \leq R_{T,\frac{1}{2}}$.

Now suppose there exists a time $t^*$ where Equation~\ref{eq:condition}
fails to hold. Hence, for some iteration, 
\[
  \eta_{t^*} =
  \frac{f(\x_{t^*}) - \tilde f}{2 ||\nabla_{t^*}||^2} \geq \frac{f(\x_{t^*}) - f(\x^{\star})}{ ||\nabla_{t^*}||^2}\,.
\]
After rearranging, we have
\[
 f(\x^{\star}) \geq \frac{f(\x_{t^*}) + \tilde f}{2 } \geq \frac{f(\xbar) +\tilde{f} }{2} =\tilde{f}_+\, .
\]
using the definition of $\xbar$ and the definition of $\tilde{f}_+$.
Hence,
$
  f(\x^{\star})-\tilde{f}_+ 
\geq 0 
$.
In addition, we have 
\begin{eqnarray*}
f(\x^{\star})-\tilde{f}_+ & =& f(\x^{\star}) - \frac{ f(\xbar) + \tilde{f}}{2}  \\
& \leq& f(\x^{\star}) -  \frac{ f(\x^*) + \tilde{f}}{2} \\
&=& \frac{f(\x^{\star})-\tilde f}{2}  
\end{eqnarray*}
which completes the proof.
\end{proof}

Now the proof Theorem~\ref{thm:simple2} follows.

\begin{proof}
  (of Theorem~\ref{thm:simple2})
  Note  $R_{T,\frac{1}{2}}\leq 2B_T$.
Suppose that
$f(\xbar_k) - f(\x^{\star})\geq R_{T,\frac{1}{2}}$ for all $k\leq K-1$, else the proof would be complete. By Lemma~\ref{lemma:shalom4},
we have that $f(\x^{\star}) - \ftil_{k}
  \leq (1/2)^{k} (f(\x^{\star})-\ftil_0)$ and that
 $0\leq f(\x^{\star}) -
    \ftil_k $, for all $k\in\{1, \ldots , K\}$.  
Hence, for $k=K-1=\lceil 2 \log \frac{f(\x^{\star})-\ftil_0}{B_T}\rceil$, we have
 $f(\x^{\star}) - \ftil_{K-1} \leq B_T$. 
   By construction  $\ftil_K=\frac{f(\xbar_{K-1})+\ftil_{K-1}}{2}$, which implies:
    \begin{eqnarray*}
    f(\xbar_{K-1}) &=& 2 \ftil_{K} - \ftil_{K-1}\\
    &\leq& 2 f(\x^{\star})- \ftil_{K-1}\\
&=&f(\x^{\star}) +f(\x^{\star})-  \ftil_{K-1} \\
&\leq& f(\x^{\star})+B_T,
\end{eqnarray*}
which completes the proof.
  \end{proof}

\subsection*{Acknowledgements}
We thank Yair Carmon for pointing out a sign error and for teaching this material.
Elad Hazan acknowledges funding from NSF award Number 1704860. 
Sham Kakade acknowledges funding from the Washington Research Foundation for Innovation in Data-intensive Discovery, the DARPA award FA8650-18-2-7836, and the ONR award N00014-18-1-2247. 

\bibliography{main.bib}
\bibliographystyle{plain}
\appendix

\section{Elementary properties of convex analysis}

\begin{lemma} \label{lem:elementary_properties}
The following properties hold for $\alpha$-strongly-convex functions and/or $\beta$-smooth functions. 
\begin{enumerate}
    \item $\frac{\alpha}{2} d_t^2 \leq h_t$
    \item $ h_t \leq \frac{\beta}{2} d_t^2$
    \item $\frac{1}{2 \beta} \|\nabla_t\|^2 \leq h_t$
    \item $ h_t \leq \frac{1}{2 \alpha} \|\nabla_t\|^2 $
\end{enumerate}
\end{lemma}

\begin{proof} 
{\bf Claim 1: $h_t \geq \frac{\alpha}{2} d_t^2 $} 
    
By strong convexity, we have 
\begin{align*}
h_t & =  f(\x_t) - f(\x^{\star}) \\
& \geq  \nabla f_t(\x^{\star}) (\x_t - \x^{\star}) + \frac{\alpha}{2} \|\x_t - \x^{\star}\|^2  \\
 & \geq \frac{\alpha}{2} \|\x_t - \x^{\star}\|^2 ,
\end{align*}
where the last inequality holds by optimality conditions for $\x^{\star}$.
    
{\bf Claim 2: $h_t \leq \beta d_t^2 $}  
    
By smoothness, 
\begin{align*}
h_t & =  f(\x_t) - f(\x^{\star}) \\
& \leq  \nabla f_t(\x^{\star}) (\x_t - \x^{\star}) + \frac{\beta}{2} \|\x_t - \x^{\star}\|^2  \\
 & \leq \frac{\beta}{2} \|\x_t - \x^{\star}\|^2 
\end{align*}
where the last inequality follows since the gradient at the global optimum is zero.

{\bf Claim 3: $h_t \geq \frac{1}{\beta} \|\nabla_t\|^2 $} 

Using smoothness: 
\begin{align*}
h_t =  & f(\x_t) - f(\x^{\star}) \\
& \geq \left\{ f(\x_t) - f(\x_{t+1}) \right\}  \\
 & \geq \left\{  \nabla f_t(\x_t) (\x_{t+1} - \x_t) - \frac{\beta}{2} \|\x_t - \x_{t+1} \|^2 \right\}  \\
 & = \eta \|\nabla_t\|^2  - \frac{\beta}{2} \eta^2 \|\nabla_t\|^2 \\
 & \geq \frac{1}{2\beta} \|\nabla_t\|^2.
\end{align*}

{\bf Claim 4: $h_t \leq \frac{1}{\alpha} \|\nabla_t\|^2 $}  
    
We have for any pair $\x,\y \in \reals^d$:
\begin{align*}
f(\y)  & \ge  f(\x) +   \nabla f(\x)^\top  (\y - \x ) + \frac{\alpha}{2}  \|\x - \y\|^2  \\
&\ge  \min_{\z \in \reals^d } \left\{ f(\x) +   \nabla f(\x)^\top  (\z - \x ) + \frac{\alpha}{2}  \|\x - \z\|^2 \right\} \\
& =  f(\x) - \frac{1}{2  \alpha} \| \nabla f(\x)\|^2. \\
& \text{ by $\z = \x - \frac{1}{ \alpha} \nabla f(\x) $ }
\end{align*}
In particular, taking $\x = \x_t \ , \ \y = \x^\star$, we have
\begin{equation*} 
 h_t =  f(\x_t) - f(\x^\star)  \leq \frac{1}{2 \alpha} \|\nabla_t\|^2  .
\end{equation*}
This completes the proof. \end{proof}

\end{document}